\documentclass[a4paper]{lms}

\usepackage{url,amsmath,amsfonts,amssymb}

\newtheorem{theorem}{Theorem}[section]
\newtheorem{lemma}{Lemma}[section]

\newcommand{\transpose}{\text{${}^{\text{T}}$}}
 \newcommand{\T}{\transpose}
\newcommand{\Z}{\mathbb{Z}} \newcommand{\C}{\mathcal{C}}

\newenvironment{ssmatrix}{\left(\begin{smallmatrix}}{\end{smallmatrix}\right)}

\newcommand{\cross}{\times} \newcommand{\h}{\widetilde}
\newcommand{\ct}{\cdot} \newcommand{\di}{\displaystyle}
\newcommand{\al}{\alpha} 
 
\newcommand{\mat}{\sim}
 \newcommand{\ti}{\times}

\title{The Hadamard circulant conjecture}
\author[B. Hurley, P. Hurley, T. Hurley]{Barry Hurley, Paul
Hurley \and Ted Hurley} \date{} %August 2007}
\makeatletter
\let\@journal\relax
\makeatother
\begin{document}

\maketitle

\begin{abstract}
It is shown that if $H$ is a circulant Hadamard $4n\ti 4n $ then $n=1$.
This proves the Hadamard circulant conjecture.\end{abstract}

\section{Introduction}
A Hadamard matrix $H$ of order $m$ is an $m\ti m $ matrix with entries
$\pm 1$ such that $HH\T = mI_m$ where $I_m$ is the identity $m\ti m$
matrix. Hadamard matrices only exist when $m=1,2$ or $m$ is a multiple
of $4$. %$m=4n$.  
They have many applications, including in coding
theory, cryptography and signal processing (details of which may be
found in~\cite{hora} or elsewhere).

The Hadamard conjecture states that there is a Hadamard $4n\ti 4n$
matrix for every $n$. For background information on this conjecture, we
refer the reader to~\cite{hora} and the references therein. The first
unknown case at present is $m=668$, the previous unknown case $m=428$
being solved in~\cite{428}.

A circulant matrix is a matrix whose rows are cycle permutations as
follows: $$\begin{pmatrix} a_0 & a_1 &a_2 & \ldots & a_{t-1} \\ a_{t-1}
& a_0 & a_1 & \ldots & a_{t-2} \\ \vdots & \vdots & \vdots &\vdots &
\vdots \\ a_{1} & a_{2} & \ldots & a_{t-1} & a_0 \end{pmatrix},$$ and a
circulant Hadamard matrix is when it is additionally Hadamard.

A $4\times 4$ Hadamard circulant matrix is one with first row containing
three elements of the same sign and another element of a different sign,
as for example in:
\begin{equation}
\begin{pmatrix}
1 &1 & 1& -1 \\ -1 & 1& 1 & 1\\ 1&-1&1&1 \\ 1&1&-1&1
\end{pmatrix}.
\label{eq:circhad}
\end{equation}
We prove the {\it circulant Hadamard matrix conjecture}, namely, that circulant
Hadamard $4n\ti 4n$ matrices do not exist for $n>1$.
\begin{theorem}{\label{hada}}
Let $H$ be an $m\ti m$ Hadamard circulant matrix. Then $m=1$ or $m=4$.
\end{theorem}

The Hadamard circulant conjecture appears to have been mentioned first
in a book by Ryser~\cite{ryser} but goes back further to obscure
sources. What was known about their existence includes:
\begin{itemize}
\item Turyn~\cite{turyn} proved that no circulant Hadamard matrices of
 orders $8p$ exist and excluded other orders of form $4(p+1)$.
\item A Hadamard circulant matrix must either have order $m = 4n^2$ or
$m=1$~\cite{sch1,sch2}.
\item There is no Hadamard circulant matrix of order $n$ with $4 < n
  \leq 10^{11}$ with three possible exceptions~\cite{sch,sch1}.
\end{itemize}
For connections between {\em Ryser's conjecture}, the circulant Hadamard
matrix conjecture and {\em Barker's conjecture} see the monograph
\cite{sch2} or paper~\cite{sch1}.
%The conjecture has significance in combinatorial design theory as a Hadamard circulant matrix of order $4n^2$ implies the
%existence of a cyclic $(4n^2,2n^2-n,n^2-n)$ design.
\subsection{Outline}
Section~\ref{sec:background} provides the necessary background for the
proof, which makes use of the connection between the group ring of the
cyclic group and circulant matrices. When the multiplicative cyclic
group $C_m$ with generator $g$ is listed in the `natural' order, namely
$\{1, g, g^2,\ldots, g^{m-1}\}$, a circulant matrix is obtained.

When a different listing is used this reduces the problem to working
with certain block matrices.
%This  gives better
% insight and makes the problem  more amenable to combinatorial
% methods.  
A different listing changes the corresponding matrix by interchanging
rows and columns and thus if the original matrix is Hadamard the matrix
relative to the new listing is also Hadamard. This new listing can be
seen in retrospect to be more `natural' %and throws light on the problem
and shows why a circulant $m\times m$ Hadamard matrix when $m > 4$
cannot exist.

As we explain later, they fail to be Hadamard because of a certain
`twist' that is introduced into the blocks formed.  By avoiding this
twist, Hadamard matrices from group ring elements may be formed.  Using
a listing in the group so that the corresponding matrix forms rows of
$2\times 2$ blocks seems to be more natural. 
%Then a further refinement
%of the listing leads to considering blocks of $4\ti 4$ matrices formed
%from the $2\ti 2$ blocks.

  This idea also leads to examples of Hadamard {\em almost circulant}
matrices (section~\ref{gmatrix}).

%\subsection{Rows of blocks}
The main part of the paper is the proof in section~\ref{sec:proof}.

%\subsection{Proof outline} 
The proof consists of the following steps:
%proceeds in the following manner:
\begin{itemize}
\item The re-derivation of the known result~\cite{sch1,sch2}, that $n$
 must be a perfect square is initially derived.
\item The elements of $C_{4n}$ are listed $\{0,2n, 1, 2n+1, \ldots,
  2n-1,4n-1\}$ where $i$ denotes $g^i$. Now form from the group ring $\Z
  C_{4n}$ the (Hadamard) $RG$-matrix relative to this listing.
\item This results in blocks $\left(\begin{smallmatrix}i & 2n+i \\ 2n+i &
  i \end{smallmatrix}\right)$ which are either {\it even} $\left(
\begin{smallmatrix} + &
  + \\ + & + \end{smallmatrix}\right)$, $\left(\begin{smallmatrix} -& - \\ - & -
  \end{smallmatrix}\right)$ or {\it odd} $\left(\begin{smallmatrix} + &
						 - \\ - & + 
  \end{smallmatrix}\right)$, $\left(\begin{smallmatrix} - & + \\ + & -
\end{smallmatrix}\right)$. There
  are then an equal number of even and odd blocks.
\item For the matrix to be Hadamard the odd blocks must occur in matching
 cancelling pairs; this is shown to be impossible for $n>1$.
\end{itemize}

%\subsection{Examples of related Hadamard matrices}
The natural question to then ask is if Hadamard matrices can arise from
other group rings. Thus, in section~\ref{gmatrix}, we provide examples
of such matrices.
%Where is the ring
These are formed within the group ring of $\Z (C_2\cross C_8)$ and
within the group ring $\Z(\mathbb{H}\cross C_2)$, where $C_t$ denotes
the cyclic group of order $t$ and $\mathbb{H}$ is the quaternion group
of order $8$. These can be extended to give Hadamard group ring matrices
over $\Z (C_2\cross C_8 \cross C_4^t)$ and over $\Z(\mathbb{H} \cross
C_2 \cross C_4^t)$ for any $t$. $C_4$ can be replaced by $C_2\ti C_2$ to
give similar examples.

\section{Background}
\label{sec:background}
%Paul: an attempt to gather the background in one section
In this section, we introduce the necessary background in group rings,
 and related matrices.
%related matrices and in difference sets in order to understand the
%proof, and for the derivations of Hadamard matrices resulting from other
%group rings in section~\ref{gmatrix}.

\subsection{Group rings and matrices}
\label{grmatrices}

Let $R$ denote a ring and $G$ a group. The {\em group ring} $RG$
consists of all (finite support) $\di\sum_{g_i\in G} \al_{g_i}g_i$ with
$\al_{g_i} \in R$. For further details on group rings see for example
\cite{seh}.

Let $G = \{g_1,g_2, \ldots, g_{n}\}$ denote the elements of a group $G$
of order $n$ and consider the elements of $G$ as being listed in this
particular order. We shall be particularly interested in the case when
$G$ is a cyclic group of order $n$ generated by $g$. The matrix of $G$,
called the $G$-matrix, (see for example \cite{hur}), relative to this
listing is $$
\begin{pmatrix}
{g_1^{-1}g_1} & {g_1^{-1}g_2} & {g_1^{-1}g_3} & \ldots & {g_1^{-1}g_n}
 \\ {g_2^{-1}g_1} & {g_2^{-1}g_2} & {g_2^{-1}g_3} & \ldots &
 {g_2^{-1}g_n} \\ \vdots & \vdots & \vdots &\vdots &\vdots \\
 {g_n^{-1}g_1} & {g_n^{-1}g_2} &{g_n^{-1}g_3} & \ldots & {g_n^{-1}g_n}
\end{pmatrix}.
$$ Suppose then $w = \di\sum_{i = 1}^n\al_{g_i}g_i \in RG$.  Then $$
w\mapsto
\begin{pmatrix}
\alpha_{g_1^{-1}g_1} & \alpha_{g_1^{-1}g_2} &\alpha_{g_1^{-1}g_3} &
 \ldots & \alpha_{g_1^{-1}g_n} \\

\alpha_{g_2^{-1}g_1} & \alpha_{g_2^{-1}g_2} &\alpha_{g_2^{-1}g_3} &
 \ldots & \alpha_{g_2^{-1}g_n} \\

\vdots & \vdots & \vdots &\vdots &\vdots \\

\alpha_{g_n^{-1}g_1} & \alpha_{g_n^{-1}g_2} &\alpha_{g_n^{-1}g_3} &
 \ldots & \alpha_{g_n^{-1}g_n}
%\ldots & \ldots & \ldots & \alpha_0 & \ldots 
\end{pmatrix}
$$ gives an embedding of $RG$ into the ring of $n\times n$ matrices over
$R$ relative to this listing of $G$.

Call the above matrix the {\em $RG$-matrix} of $w$ relative to the given
listing and call a matrix obtained from the group ring $RG$ in this
manner an $RG$-matrix. The term {\em $RG$-matrix} is used here as the
methods are basically group ring methods.

%Thus $M(RG,w)$ is in $R_{n\times n}$. 

The structure of an $RG$-matrix may clearly be seen in the following
table:
\begin{equation}\left(\begin{array}{l|lllll}& g_1 & g_2 & g_3 & \ldots & g_n \\
\hline g_1^{-1} & \alpha_{g_1^{-1}g_1} & \alpha_{g_1^{-1}g_2} &
\alpha_{g_1^{-1}g_3} & \ldots & \alpha_{g_1^{-1}g_n} \\ g_2^{-1}
&\alpha_{g_2^{-1}g_1} & \alpha_{g_2^{-1}g_2} &\alpha_{g_2^{-1}g_3} &
\ldots & \alpha_{g_2^{-1}g_n} \\

\vdots & \vdots & \vdots & \vdots &\vdots &\vdots \\

g_n^{-1} & \alpha_{g_n^{-1}g_1} & \alpha_{g_n^{-1}g_2}
 &\alpha_{g_n^{-1}g_3} & \ldots & \alpha_{g_n^{-1}g_n}
%\ldots & \ldots & \ldots & \alpha_0 & \ldots 
\end{array}\right)
\label{kump}
\end{equation}
The first column is essentially {\em labelled} by $g_1$, the second by
$g_2$ etc.; the first row is {\em labelled} by $g_1^{-1}$, the second by
$g_2^{-2}$ etc.  As it aids in the calculation of an $RG$-matrix
relevant to a particular listing, this formulation will be called upon
later. %referred to

\subsubsection{Cyclic group rings, circulant matrices, and alternate listings.}
When $G=C_n = \{1,g,g^2,\ldots,g^{n-1}\}$ is the cyclic group of order
$n$, we get a group ring $R C_n$.

Let $RC_{n\times n}$ denote the ring of circulant $n\times n$ matrices
which is a subring of $R_{n\times n} $, the ring of $n\times n$ matrices
over $R$.

Then $RC_{n\ti n} \cong RC_n \cong \frac{R[x]}{\langle x^n - 1 \rangle}$
where $R[x]$ is the polynomial ring in variable $x$ over $R$ and
$\langle x^n - 1 \rangle$ is the ideal in $R[x]$ generated by $x^n -1$.

The isomorphism makes it possible to move from the circulant matrices to
the cyclic group ring and back, exploiting properties of both.

Let $w = \sum_{i=0}^{n-1}\alpha_ig^i$. Then the $RG$-matrix of $w$
relative to this natural listing is the circulant matrix $$
\begin{pmatrix}
\al_0 & \al_1 &\al_2 & \ldots & \al_{n-1} \\ \al_{n-1} & \al_0 & \al_1 &
  \ldots & \al_{n-2} \\ \vdots & \vdots & \vdots &\vdots & \vdots \\
  \al_{1} & \al_{2} & \ldots & a_{n-1} & \al_0
\end{pmatrix}.
$$ Suppose, in the general case, the listing is changed. Any change of
listing corresponds to a permutation of the elements of the group and a
permutation is generated by transpositions. If the elements $g_i$ and
$g_j$ are interchanged then the new matrix is obtained from the original
matrix by interchanging the $i^{th}$ and $j^{th}$ rows and then
interchanging the $i^{th}$ and $j^{th}$ columns.  Thus if an $RG$-matrix
relative to some listing is Hadamard then the $RG$-matrix relative to
any listing, as one might expect, is also Hadamard.

For example, let $G$ be the cyclic group $C_4$ generated by $g$ with $w =
1+g+g^2-g^3$.

 Then relative to the natural ordering $w$ has the Hadamard
circulant matrix $\begin{ssmatrix}1 & 1 & 1 & -1 \\ -1& 1 & 1 & 1 \\ 1 &
-1 & 1& 1 \\ 1& 1 & -1 & 1
\end{ssmatrix}$.

When $G$ has the listing $G=\{1,g^2,g,g^3\}$, $w$ has the Hadamard
matrix $\begin{ssmatrix} 1 & 1 &| & 1 & -1 \\ 1& 1 &| & -1 & 1 \\ \hline -1 &
1&| & 1& 1 \\ 1& -1 &|& 1 & 1
       \end{ssmatrix}$. 

The new listing is obtained by interchanging $g$ and $g^2$ in the
       original listing, and thus the new matrix is obtained from the
       original by interchanging the second and third rows and then the
       second and third columns.
Notice how the `odd' block  $\begin{ssmatrix}1 & -1 \\ -1 & 1\end{ssmatrix}$ 
of this new matrix gets
       `twisted' to a new type of `odd' 
 block $\begin{ssmatrix} -1 & 1 \\ 1 & -1
			       \end{ssmatrix}$ in 
 the next row of blocks as indicated; in this small case however this is
			not a problem as the even block still acts above
			and below the `odd' blocks.

\section{Proof of the conjecture}
\label{sec:proof}
%\section{Proof of Theorem \ref{hada}}

Suppose now that $H$ is a circulant Hadamard $4n\times 4n$
 matrix. Let  $C_{4n}$ be generated by $g$ 
 and
let $w$ be the element in the cyclic group ring $\mathbb{Z} C_{4n}$
corresponding to $H$. 
The inner product of any two (different) rows of a Hadamard matrix is
$0$.

\subsection{$n$ must be a perfect square.}

%Let $H$ be a circulant Hadamard $4n\times 4n$ matrix. 
 Suppose there are
$r$ negative entries in the first row of $H$ and consequently there are $r$
negative elements in each row and column.  There are then $(4n-r)$
positive entries in each row and column.

It is known that for any existing circulant $4n\times 4n$ Hadamard
matrix $n$ must be a perfect square~\cite{sch1,sch2}. For completeness,
the following verifies this result and additionally provides information
on the number of possible negative and positive entries.
%Although the result is known, a proof is given for completeness.
%The method of proof reflects later types of arguments.
%Is this exact result known?
\begin{lemma}\label{just} $r= 2n\pm \sqrt{n}$.
\end{lemma}
\begin{proof} The sum of the elements in any column is $4n-r-r = 4n-2r$.

The sum of the elements in a column except a positive element is thus
 $4n-2r-1$ and the sum of the elements in a column except a negative
 element is $4n-2r+1$. The rows of the matrix are orthogonal to one
 another and thus the sum of the first $4n-1$ rows is perpendicular to
 the last row.  The last row has $4n-r$ positive elements and $r$
 negative elements. Thus $(4n-r)(4n-2r-1) - r(4n-2r+1) = 0$.

This implies that $$ 16n^2 -8nr -4n -4nr+2r^2 + r - 4nr + 2r^2 - r = 0.
$$ Hence $16n^2 - 16nr + 4r^2 -4n = 0$ from which $r^2-4nr + 4n^2 - n =
0$.  Solving this quadratic for $r$ gives: $$ r = \frac{4n \pm
\sqrt{16n^2 - 16n^2 + 4n }}{2} $$ and so $r=2n\pm \sqrt{n}$, as
required.
\end{proof}
It then follows that the number of positive elements is $4n-r = 2n\mp
\sqrt{n}$, as expected from symmetry.

The result of Lemma \ref{just} depends only on the fact that the same
number of positive elements appears in each row and column. A Hadamard
matrix whose row and column sums are the same is known as a {\em regular
Hadamard matrix}. Thus whenever $RG$-matrices (which include circulant
matrices) are Hadamard, they are by necessity also regular Hadamard
matrices.
%Thus $RG$-matrices, which includes circulant matrices,
%are regular Hadamard matrices.

The following may be proved along the lines of Lemma \ref{just}.
\begin{lemma}\label{just1} Let $A$ be a $m\times m$ matrix with orthogonal rows 
consisting of $\pm 1$ entries in which the same number $m-r$ of $+1$ and
 the same number, $r$, of $-1$ appear in each row and column. Then
 $r=\frac{m\pm \sqrt{m}}{2}$. \endproof
\end{lemma}

\subsection{New listing: Pairs in blocks}

The elements of the cyclic group $G=C_{4n}$ are $\{1, g, g^2 , \ldots,
g^{4n-1}\}$, listed in the `natural' ordering. Write $i$ for $g^i$ so
this listing is $\{0, 1,2, \ldots, 4n-1\}$.

Now $H$ is a Hadamard circulant matrix obtained from the
$RG$-matrix corresponding to an element of $ \Z G$ 
relative to  this natural ordering. An interchange of rows or columns in a Hadamard matrix results in
another Hadamard matrix. If two group elements are interchanged in the
listing then the resulting $RG$-matrix is obtained by interchanging two
rows and then interchanging two columns and the result is thus a
Hadamard matrix. Any new listing may be obtained by a succession of
interchanges and thus if the listing is changed the new $RG$-matrix
obtained is still a Hadamard matrix.
 
List the elements of $G$ as follows: $\{0, 2n, 1, 2n+1, \ldots,
2n-1,4n-1\}$. Using this ordering breaks the $RG$-matrix into blocks of
$2 \ti 2$ matrices and rows of these blocks.  The inner product of any
two (different) rows is $0$ as the matrix is Hadamard.
 %Now exploit the fact that the inner product of any  two (different)
%rows in a block  must be  $0$ and then any row in a block must have
%inner product $0$ with a row in another block.  
% which interact with one another. 

Let $\h{i}$ denote the coefficient of $g^i$ in the $RG$-matrix; each
$\h{i}$ is either $+ 1$ or $-1$.

Consider now the $RG$-matrix with this listing as constructed in
\eqref{kump} of Section \ref{grmatrices}. The first two rows are then: $$
\begin{pmatrix} \h{0} &\h{2n}& \h{1}& \h{2n+1}&
  \ldots & \ldots & \h{2n-1}& \h{4n-1} \\ \h{2n}& \h{0} & \h{2n+1} &
    \h{1} & \ldots & \ldots & \h{4n-1} & \h{2n-1} \end{pmatrix} $$ The
    inner product of these two rows gives that $2\{(\h{0}\ti \h{2n}) +
    \h{1}\ti (\h{2n+1}) + \h{2}\ti \h{(2n +2)} + \ldots + (\h{2n-1)}\ti
    \h{(4n-1)})\} = 0$. Hence $\h{0}\ti \h{2n} + \h{1}\ti \h{(2n+1)} +
    \h{2}\ti \h{(2n+2)}+ \ldots + \h{(2n-1)}\ti \h{(4n-1)} = 0$.

Thus one half of the products $\h{i}\ti \h{(2n+i)}$ , $0\leq i \leq
2n-1$, are $+1$ and one half are $-1$.  Hence $(\h{i},\h{2n+i})$ for
$i=0,\ldots 2n-1$ consist of $(+1,-1)$ or $(-1,+1)$ for $n$ times and
$(+1,+1)$ or $(-1,-1)$ for the other $n$ times.

Consider in more detail the $G$-matrix obtained from the listing
$\{0,2n, 1, 2n+1 ,\ldots, 2n-1, 4n-1$\} $$
\begin{pmatrix} 0 & 2n & 1& 2n+1 & 2& 2n + 2 & \ldots
  & \ldots & 2n-1 & 4n-1\\ 2n&0 & 2n+1&1 & 2n+2&2& \ldots & \ldots &
    4n-1& 2n-1 \\\hline 4n - 1 & 2n-1 & 0 & 2n & 1& 2n +1 & \ldots
    &\ldots & 2n-3 & 4n-3 \\ 2n-1&4n-1& 2n&0&2n+1&1 & \ldots &\ldots &
    4n-3& 2n-3 \\\hline \vdots & \vdots & \vdots & \vdots & \vdots
    &\vdots&\vdots&\vdots &\vdots & \vdots \end{pmatrix} $$ The elements
    of $G$ are replaced by their coefficients to obtain the $RG$-matrix
    which is assumed to be Hadamard.

It is convenient to write $i$ for the coefficient $\h{i}$ of $i = g^i$
in the $RG$ matrix. Thus when considering the $G$-matrix $i=g^i$, and
when considering the $RG$-matrix $i$ denotes the coefficient of $g^i$.

For example on considering the above matrix as an $RG$-matrix the inner
product of the first row with the third row gives that $$ (0,2n)\ct
(4n-1,2n-1) + (1,2n+1)\ct (0,2n) + \ldots + (2n-1,4n-1)\ct (2n-3,4n-3) =
0 $$ where $\ct$ denotes the inner product of the length $2$ vectors.

The $RG$-matrix formed -- see \eqref{kump} above -- with this listing is
now considered in more detail.  The matrix has been broken into $2\ti 2$
blocks for detailed analysis. Here $i$ in the first row or column means
the group element $g^i $, and $i$ elsewhere is the coefficient of $g^i$
which is either $+1$ or $-1$.
 %Consider then the following obtained from the listing $0,2n,2,2n+2,
 %\ldots 2n-2, 4n-2, 1, 2n + 1, 3, 2n+3, \ldots, 2n-1,4n-1$   
 %
$$\begin{array}{c||cc|cc|cc|c|cc} & 0 & 2n &1& 2n+1& 2&2n+2& \ldots
 &2n-1& 4n-1\\ \hline \hline 0& 0 & 2n & 1& 2n+1 &2&2n+2& \ldots &2n-1&
 4n-1 \\ 2n & 2n&0& 2n+1&1& 2n+2& 2 & \ldots & 4n-1 & 2n-1 \\ \hline
 4n-1 & 4n-1& 2n-1&0& 2n& 1& 2n+1& \ldots& 2n-2&4n-2 \\ 2n -1&2n-1 &
 4n-1 & 2n &0 & 2n+1& 1 &\ldots & 4n-2& 2n-2 \\ \hline \vdots &\vdots &
 \vdots& \vdots& \vdots & \vdots &\vdots & \vdots & \vdots &\vdots \\
 \hline 2n+1 & 2n+1 & 1 & 2n+2& 2 & 2n+3& 3 & \ldots & 0&2n \\ 1 & 1 &
 2n+1 & 2 & 2n+2& 3 & 2n+ 3& \ldots & 2n&0 \\ \hline
\end{array}$$

Notice that the last block in a row of blocks is `twisted' when it
appears as the first block of the next row of blocks, as for example
$\begin{pmatrix} 2n-1 & 4n-1 \\ 4n-1 & 2n -1 \end{pmatrix}$ in the first
row of blocks becomes $\begin{pmatrix} 4n-1 & 2n-1 \\ 2n-1 & 4n-1
\end{pmatrix}$ in the second row of blocks.
%\vspace{.1in} 

The inner product of any two rows must all be $0$ since the matrix is
Hadamard. The inner product of a row of a block with another row of that
block has ensured that there are an equal number of even and odd pairs
in each row of blocks.

\subsection{Even and odd blocks}
Suppose $B = \left(\begin{smallmatrix} i & j \\ j &i \end{smallmatrix}\right)$
where $i=\pm 1, j=\pm 1$. Call such a $B$ a {\em $2$-block} or simply a
{\em block} when the size is clear. 
Say that $B$ is {\em even}
when $i=j$ and {\em odd} when $i\not = j$. Define $\h{B} =
\left(\begin{smallmatrix} j&i\\i& j \end{smallmatrix}\right)$. Thus
$\h{B} = -B$ when $B$ is odd and $\h{B}=B$ when $B$ is even.

  Let $m = 2n$.
The matrix has been broken into $2$-blocks as follows:
$$M=\left(\begin{array}{llllll} B_1 & B_2 & B_3 & \ldots &\ldots & B_m \\
	\h{B_m} & B_1 & B_2 &\ldots &\ldots & B_{m-1} \\ \h{B_{m-1}} & \h{B_m}
	  & B_1 & \ldots &\ldots & B_{m-2} \\ \vdots& 
\vdots & \vdots & \vdots & \vdots &
	  \vdots  \\ \h{B_2} & \h{B_3} & \h{B_4} & \ldots & \h{B_m} & B_1
	\end{array}\right)$$

The inner product of two rows of  blocks  of $M$ is the  
sum $\sum B_iB_j\T = \sum B_iB_j$ over each $B_i$ in  one of the 
rows of blocks with its corresponding $B_j$ in the other row of blocks. 
As $M$ is
Hadamard the inner product of two (different) rows of blocks in 
$M$ must be the zero matrix $0= 0_{2\times 2}$. 

We can assume that $n \geq 4$ by Lemma \ref{just}.

Now $B_iB_j = 0$ when one of $B_i,B_j$ is even and the other is odd, $B_iB_j
= \pm \left(\begin{smallmatrix} 2&2\\ 2&2 \end{smallmatrix}\right)$ 
when both $B_i,B_j$ are even
  and $B_iB_j= \pm \left(\begin{smallmatrix} 2&-2\\ -2 &
    2\end{smallmatrix}\right)$ when both $B_i,B_j$ are odd.
Thus in any  sum of products of $2$-blocks in an equation 
$\sum B_iB_j = 0$ the part of the sum involving products between  
even $2$-blocks
and the part of the sum involving  products between   
odd $2$-blocks must both be $0$.  

There are $n$ even blocks and $n$ odd blocks in the system 
$S = \{B_1, B_2, \ldots, B_m\}$ of $2$-blocks of $M$. Denote block $B_i$ by $i$
when convenient and unambiguous.  The blocks $i,j$ are said to be of 
 the {\em same kind} if they are both even or both odd.  
Define the 
{\em difference} of $(B_i,B_j)$ to be $j-i$ for $i<j$ and to be 
$2n-(j-i)$ for $i>j$. Call $(B_j,B_i)$ the {\em
  conjugate} of $(B_i,B_j)$. Denote the difference of $(i,j)$ by
  $d(i,j)$. The differences of $S$ consist of all differences $d(i,j)$
  with $i,j\in S, i\neq j$. 

Say the block $i$ is {\em symmetric} if there exists a block of the same kind
at difference $n$ away, and otherwise say the block $i$  {\em is not
symmetric} or {\em is non-symmetric}. When $i$ is symmetric denote by  
$i^{'}$ the block (of the same kind) satisfying $d(i,i^{'}) = n =
d(i^{'},i)$.
  
 Define the {\em sign} of the even $2$-block $\left(\begin{smallmatrix} 1&1 \\
1&1\end{smallmatrix}\right)$ to be $+1$ and define the sign of the other
even $2$-block to be $-1$. 
Define the {\em sign} of  the odd $2$-block
 $\left(\begin{smallmatrix}1& -1 \\ -1 & 
1\end{smallmatrix}\right)$ to be  $+1$ and
the sign of the other odd $2$-block to be $-1$.

Define the sign of a pair $(i,j)$ of blocks as
follows. When $i,j$ are even, define the sign of $(i,j)$
to be $+1$ if both $i,j$
have the same sign and to be  $-1$ otherwise. When $i,j$ are 
odd and $i<j$ define the sign of $(i,j)$ to be $+1$ if both
$i,j$ have the same sign and to be $-1$ otherwise and    
define the sign $(j,i)$ (for $j>i$) to be the opposite to that of
the sign of $(i,j)$. 
If  any difference $t$ between odd $2$-blocks 
occurs an odd number of times  in differences of $S$  then the (inner)
product of the first  
 row of blocks of $M$ with the $(t+1)^{st}$ row of blocks of $M$ cannot be $0$.
Thus every difference between the odd blocks 
must occur an even number of times.  
Every pair of odd  blocks must have a matching pair 
of odd blocks with equal difference and opposite sign. Similarly 
every difference between the even $2$-blocks must occur
an even number of times and every pair of even blocks must have a
matching pair of even blocks with equal difference and opposite
sign. 

Consider now only pairs $(i,j)$ of blocks with $i\neq j$ where 
$i,j$ are of the same kind.

Let $i,j,k,l$ be blocks of the same kind with $i\neq j, k\neq l$. 
Then $d(i,j)=d(k,l)$ if and only if $d(j,i)=d(l,k)$. 
Say the pair of blocks $(i,j)$ {\em is a match of} the pair  $(k,l)$ 
if and only if $d(i,j)=d(k,l)$ 
and the pairs have opposite signs; in this case say $(i,j)$ {\em
  matches} $(k,l)$ or that $(k,l)$ is a {\em matching pair} for
$(i,j)$.  Then $(i,j)$ matches $(k,l)$ if and
only if  $(j,i)$ matches $(l,k)$. Use $(i,j)\mat
(k,l)$ to mean that $(i,j)$ matches $(k,l)$ and $(i,j) 
\not \mat (k,l)$ to mean that $(i,j)$ does not match $(k,l)$.
Every pair has a unique matching pair and say $(i,j)$ {\em
is matched with} $(k,l)$ if $(i,j) \mat (k,l)$ match and $(i,j)$ has been
uniquely assigned to $(k,l)$. It is implicitly 
assumed that if $(i,j)$ is matched with $(k,l)$
then the conjugates $(j,i)$ and $(l,k)$ are also matched with one another.

Let $i,j$ be symmetric odd blocks.  Then $(i,i^{'}) \mat
 (i^{'},i)$ and $(j,j^{'}) \mat (j^{'},j)$.  Also
 $d(i,j)=d(i^{'},j^{'})$  and $d(j,i^{'})=
d(j^{'},i)$ and {\em either} $(i,j) \mat (i^{'},j^{'}), (j,i)\mat
 (j^{'},i^{'})$ and $(j,i^{'}) \not \mat (j^{'},i), (i^{'},j) \not
 \mat (i,j^{'})$ {\em or else} $(i,j)\not \mat (i^{'},j^{'}),(j,i) \not \mat
 (j^{'},i^{'})$ and   $(j,i^{'}) \mat (j^{'},i), (i^{'},j)\mat  
 (i,j^{'})$. The pairs designated as non-matching have equal
 differences and the same sign. 

Let $i,j$  be symmetric even blocks.  Then
$d(i,i^{'})=d(j,j^{'})=n, d(i,j)=d(i^{'},j^{'}), d(j,i^{'})=
d(j^{'},i)$.  Also $(i,i^{'}) \mat (j,j^{'})$ if and only if 
$(i,j) \mat (i^{'},j^{'})$ if and only if $(j,i^{'})\mat (j^{'},i)$.
Now $(i,i^{'}) \mat (j,j^{'})$  if and only if the pairs 
have opposite signs. If $(i,i^{'})$ is matched with
$(j,j^{'})$ then match  $(j,i^{'})$ with $(j^{'},i)$ and
match $(i,j)$ with $(i^{'},j^{'})$.

Say a pair $(i,j)$ {\em is not balanced} or {\em is
  unbalanced} if either $i$ or $j$ is not symmetric.
Let $i$ be an even non-symmetric
block if one such exists. Consider  a non-balanced pair $(i,j)$. If 
$(i,j)\mat (j,k)$ for some $k$ then $(i,k)$ is a unbalanced pair and is
  matched elsewhere. Suppose $(i,j)\mat (k,i)$.  
When $k$ is not symmetric this gives an unbalanced pair $(k,j)$
which is matched elsewhere. 
When $k$ is symmetric two cases can arise.
Firstly if $k=j^{'}$ then $(j,j^{'})$, which has distance $n$, is matched by a
pair $(l,l^{'})$ for a symmetric $l$
and this gives a new unmatched pair $(i,l)$ which is not
balanced and is matched elsewhere. 
Secondly if $k\neq j^{'}$, a new unbalanced pair
$(i,k^{'})$ is obtained which is matched elsewhere. 
Suppose $(i,j) \mat (k,l)$ where none of the indices are equal. Then
$(i,k)\mat (j,l)$ match and $d(j,k)\neq d(i,l), d(j,k)\neq d(l,i)$ 
as $i$ is not symmetric. Hence $(i,l)\not \mat (j,k)$ and $(i,l)\not
\mat (k,j)$. This gives the new unbalanced
pair $(i,l)$ which has a match elsewhere. 
In all cases a new unbalanced pair is derived which must be 
matched elsewhere. Continue this matching process. However this cannot 
continue indefinitely. Thus if a a non-symmetric even $2$-block exists 
there is a pair of even $2$-blocks which is not matched.  
Hence every even block is symmetric. It follows that every odd block
  is symmetric.    

Since all the blocks are  symmetric, $n$ must be even and
let $n=2q$. 
Consider now the odd $2$-blocks.
These may then be  listed as $\{i_1,i_2,\ldots,i_q,
i_1^{'},i_2^{'}, \ldots, i_q^{'}\}$, where $i_j$ means the block $B_{i_j}$
and the indices are ordered $i_1< i_2 <\ldots <i_q$.  % only.
 Say $j\in T$ if $j$ is one of $\{i_1,i_2,\ldots, i_q\}$. 

Suppose first of all that $n=4$. Then the odd $2$-blocks are
$\{i,j,i^{'},j^{'}\}$. If $(i,j)\mat (i^{'},j^{'})$ then
$(j,i^{'})\not \mat (j^{'},i)$ and
if $(j,i^{'})\mat (j^{'},i)$  then $(i,j)\not \mat (i^{'},j^{'})$.
 This is impossible.

Suppose then $n>4$.
For each $i,j\in T$ consider the pairs in $T(i,j) = \{i,j,i^{'},j^{'}\}$
and match these where possible. As already noted $(i,i^{'}) \mat
(i^{'},i), (j,j^{'})\mat (j^{'},j)$. Either $(i,j)\mat (i^{'},j^{'})$ and
$(j,i^{'})\not \mat (j^{'},i)$ 
or else $(j, i^{'})\mat (j^{'},i)$  and $(i,j) \not \mat (i^{'},j^{'})$. 
Call the non-matching pairs, $(i,j),(i^{'},j^{'})$ or
$(j,i^{'}),(j^{'},i)$ as appropriate, and their conjugates the 
{\em remainders} from
$T(i,j)$ and denote these by $rem(T(i,j))$. Other pairs from  $T(i,j))$ 
besides those in $rem(T(i,j))$ are matched. The remainders occur in pairs
which have equal differences and same sign together with their conjugates
which must also have equal differences and same sign. Clearly
$T(i,j)=T(j,i))$ and so in specifying $T(i,j)$ we can assume from now on
that $i<j$. For  $ i<j$ the elements in $rem(T(i,j))$ 
with differences less than $n$ 
are either $(i,j),(i^{'},j^{'})$ or $(j,i^{'}),(j^{'},i)$ and the
other remainders are the conjugates of these with differences greater
than $n$. 

If one remainder from $T(i,j)$ matches a remainder 
from $T(k,l)$ then all the remainders from $T(i,j)$ can be matched with
remainders from $T(k,l)$ and in this case say $rem(T(i,j))$ match $rem(T(k,l))$.
If $rem(T(i,j))$ match $rem(T(k,l))$ then $rem(T(k,l))$ match
$rem(T(i,j))$.

Suppose $rem(T(i,j))$ match $rem(T(j,k))$ or $rem(T(k,i))$ for
 some $k$. In this case say $rem(T(i,j))$ are matched
by the remainders from  an {\em adjacent} pair. 
Consider the case that $rem(T(i,j))$ match  
 $rem(T(j,k))$; the other case is similar. 
Then $d(i,j)=d(j,k)$ and $rem(T(i,k))$ match remainders from another pair. 
If $rem(T(i,k))$ are matched by the remainders from  an
adjacent pair then this gives further remainders $rem(T(i,r))$ or $rem(T(r,k))$
 matched by the remainders from some other $T(p,q)$.
Continue in this way until eventually  $T(p,q)$
is obtained whose remainders are not matched by remainders from 
 an adjacent pair. 
%Match  the
%remainders of any pair by remainders of an adjacent pair  
%until a pair is reached whose remainders  are not 
%matched by an adjacent pair. 

Suppose $rem(T(i,j))$ match $rem(T(k,l))$ where $i,j,k,l$ are distinct
and we may assume that $i<k$. 
Then $d(i,j) = d(k,l)$. (Note that $i<j,k<l$ and that $d(i,j) =
d(k,l)$ if and only if $d(j,i^{'}) = d(l,k^{'})$.) Then either $k<j$ or
$k>j$; assume $k>j$ and the other case is similar. 
Then $d(i,k)= d(j,l)$ and $d(i,l)= d(k,j^{'})$.  
It follows that {\em either} $rem(T(i,l))$ match
$rem(T(j,k))$ and $rem(T(i,k)),rem(T(j,l))$ consist of two pairs from each  
with equal differences and same sign together with their conjugates
 {\em or else} $rem(T(i,k))$ match $rem(T(j,l))$ and
$rem(T(i,l)), rem(T(k,j))$ consist of  two pairs from each with
equal differences and  same sign together with their conjugates.
Match all remainders possible 
resulting from the match of $rem(T(i,j))$ and $rem(T(k,l))$.
There still remains {\em either}  $rem(T(i,k))$ and $rem(T(j,l))$ {\em or}
$rem(T(i,l))$ and $rem(T(k,j))$ which are matched by remainders 
from different pairs. If these remainders can be matched by remainders
from an adjacent pair this leads to  
a pair whose remainders cannot be matched by remainders of 
an adjacent pair as already pointed out. 
Continue matching in this way. At each stage a new pair is obtained whose
remainders are matched by the remainders of another pair not already matched. 
This process cannot continue indefinitely and hence there is
a pair of odd $2$-blocks which is not matched. 
 
Thus the
matrix $M$ is not Hadamard for $n>1$ and thus the original matrix is not
Hadamard as required. 
This completes the proof that there is no Hadamard $4n\ti 4n$ matrix for
$n>1$. 

\section{Hadamard matrices from group rings}\label{gmatrix}
%In the case where $n$ is even and the pairs
%$(+,+),(-,-)$ interact with one another and the pairs $(+,-),(-,+)$ 
%interact with one another,  
We have shown that no circulant Hadamard matrices of size $4n\ti 4n$ for
$n>1$ can exist. Equivalently, there is no Hadamard $RG$-matrix from the
group ring $ \Z C_{4n}$-matrix except when $n=1$.  For the case $n=1$,
we showed in (\ref{eq:circhad}) an example of a Hadamard matrix from the
group ring $\Z C_4$.

A natural question arises: From other groups $G$, when do Hadamard
matrices exist?  Lemma~\ref{just1} tells us that in general it is
necessary for the group to be of order $4u^2$. %$|G| = 4u^2$.

We now provide a couple of illustrations of these matrices, first from
groups of direct products of cyclic groups, and then using a direct
product of the quaternion group with cyclic groups.

The following is a Hadamard $\Z (C_2\cross C_2)$-matrix:
$\begin{pmatrix} 1 &1 &| & 1& -1 \\ 1 & 1&|& -1 & 1 \\ \hline 1&-1&|&1&1 \\
-1&1&|&1&1 \end{pmatrix}$.

In the cyclic (circulant) case a Hadamard matrix is not obtained because
of a `twist' in one of the pairs $(+,-),(-,+)$ from the end of one block
to the beginning of the next.  By avoiding this `twist' a Hadamard
matrix that is `almost' circulant may be obtained.

Consider thus the case $n=16$ where the first row is
$$\begin{array}{rrrrrrrr} ++ & +- & ++ & +- & ++ & +-& -- & -+
\end{array}$$ and then proceed by avoiding the `twist' that a circulant
would introduce to get the following Hadamard matrix: $$
\left(\begin{array}{r|r|r|r|r|r|r|r} 
++ & +- & ++ & +- & ++ & +-& -- &-+ \\ ++ & -+ & ++ & -+ & ++ & -+& --
&+- \\ \hline -+& ++ & +- & ++ & +- & ++ & +-&-- \\ +-& ++ & -+ & ++ &
-+ & ++ & -+&-- \\ \hline --&-+& ++ & +- & ++ & +- & ++ & +- \\ --&+-&
++ & -+ & ++ & -+ & ++ & -+ \\ \hline +-&--&-+& ++ & +- & ++ & +- & ++
\\ -+&--&+-& ++ & -+ & ++ & -+ & ++ \\ \hline ++&+-&--&-+& ++ & +- & ++
& +- \\ ++&-+&--&+-& ++ & -+ & ++ & -+ \\ \hline +-&++&+-&--&-+& ++ & +-
& ++ \\ -+&++&-+&--&+-& ++ & -+ & ++ \\ \hline ++&+-&++&+-&--&-+& ++ &
+- \\ ++&-+&++&-+&--&+-& ++ & -+ \\ \hline +-&++&+-&++&+-&--&-+& ++ \\
-+&++&-+&++&-+&--&+-& ++
 \end{array}\right)
 $$ The above example is actually a Hadamard $RG$-matrix from the group
ring $\Z(C_2\cross C_8)$. This may be extended in an obvious way to form
Hadamard $RG$-matrices from $\Z( C_2\cross C_8 \cross C_4^t)$ for any
$t$. $C_4$ may be replaced by $C_2\cross C_2$ to give further examples.

Looking at how the twists work as in the proof of Theorem~\ref{hada}, a
Hadamard $RG$-matrix from the group ring $\Z (\mathbb{H}\cross C_2)$ may be
constructed as follows, where $\mathbb{H}$ is the quaternion group of
order $8$. This is an example of a Hadamard $RG$-matrix where $G$ is
non-commutative.  $$
\left(\begin{array}{r|r|r|r||r|r|r|r} 
++ & +- & ++ & +- & ++ & -+& -- &+- \\ ++ & -+ & ++ & -+ & ++ & +-& --
&-+ \\ \hline -+& ++ & -+ & ++ & +- & ++ & -+&-- \\ +-& ++ & +- & ++ &
-+ & ++ & +-&-- \\ \hline ++&+-& ++ & -+ & -- & +- & ++ & +- \\ ++&-+&
++ & +- & -- & -+ & ++ & -+ \\ \hline -+&++&+-& ++ & -+ & -- & -+ & ++
\\ +-&++&-+& ++ & +- & -- & +- & ++ \\ \hline \hline ++&-+&--&+-& ++ &
+- & ++ & +- \\ ++&+-&--&-+& ++ & -+ & ++ & -+ \\ \hline +-&++&-+&--&-+&
++ & -+ & ++ \\ -+&++&+-&--&+-& ++ & +- & ++ \\ \hline
--&+-&++&+-&++&+-& ++ & -+ \\ --&-+&++&-+&++&-+& ++ & +- \\ \hline
-+&--&-+&++&-+&++&+-& ++ \\ +-&--&+-&++&+-&++&-+& ++
\end{array}\right)
$$ There are `twists' here which fit together to cancel one another.
This may be extended to examples of Hadamard $RG$-matrices from $\Z(
\mathbb{H} \cross C_2 \cross \C_4^t)$ in an obvious way. $C_4$ may be
replaced by $C_2\cross C_2$ to give similar examples.

\affiliationone{% in this example, two authors share an institution
   Barry Hurley and Ted Hurley\\ National University of Ireland Galway\\
   Galway\\ Ireland.  \email{barry.hurley@nuigalway.ie\\
   ted.hurley@nuigalway.ie}}
% Important: Do not put any empty line here.
\affiliationtwo{Paul
   Hurley\\ IBM Research Zurich\\ S\"aumerstrasse 4, CH-8803
   R\"uschlikon\\ Switzerland.  \email{pah@zurich.ibm.com} }

\end{document}